\documentclass[12pt,oneside,reqno]{amsart}

\usepackage{tikz}
\usetikzlibrary{arrows.meta}
\usetikzlibrary{bending}

\usepackage{url}

\usepackage{hyperref}
\usepackage{amsxtra}
\usepackage{amsopn}
\usepackage{amsmath,amsthm,amssymb}
\usepackage{color}
\usepackage{amscd}
\usepackage{pifont}
\usepackage{amsfonts}
\usepackage{latexsym}
\usepackage{verbatim}
\usepackage{pb-diagram}
\usepackage{tikz}

\newcommand{\fr}{\mathfrak}

\newcommand{\op}{\operatorname}

 \newtheorem{lemma} {Lemma} [section]
\newtheorem{theorem}[lemma]{Theorem} 
\newtheorem{remark}[lemma] {Remark} 
\newtheorem{prop} [lemma]{Proposition}  
\newtheorem{definition}[lemma] {Definition} 
\newtheorem{corol}[lemma] {Corollary} 
\newtheorem{example}[lemma] {Example}

\begin{document}
\title{Geodesic orbit metrics on the compact Lie group $G_2$} 

\author{Nikolaos Panagiotis Souris}
\address{University of Patras, Department of Mathematics, University Campus, 26504, Rio Patras, Greece}
\email{nsouris@upatras.gr}

\begin{abstract} Geodesics on Riemannian manifolds are precisely the locally length-minimizing curves, but their explicit description via simple functions is rarely possible. Geodesics of the simplest form, such as lines on Euclidean space and great circles on the round sphere, usually arise as orbits of one-parameter groups of isometries via Lie group actions.  Manifolds where all geodesics are such orbits are called geodesic orbit manifolds (or g.o. manifolds), and their understanding and classification spans a quite long and continuous history in Riemannian geometry. In this paper, we classify the left-invariant g.o. metrics on the compact Lie group $G_2$, using the nice representation theoretic behaviour of a class of Lie subgroups called (weakly) regular. We expect that the main tools and insights discussed here will facilitate further classifications of g.o. Lie groups, particularly of lower ranks.

\medskip
\noindent  {\it Mathematics Subject Classification 2020.} Primary 53C25; Secondary 53C30. 

\medskip
\noindent {\it Keywords}:  geodesic orbit Lie group; geodesic orbit metric.

\end{abstract}
\maketitle

\section{Introduction}

A geodesic $\gamma$ on a Riemannian manifold $(M,g)$ is called \emph{homogeneous} if it is generated by isometries, in the sense that there exists a Lie group $G$ of isometries of $(M,g)$ such that 

\begin{equation}\label{homgeod}\gamma(t)=\exp(tX)\cdot p, \ \ t\in \mathbb R,\end{equation}

\noindent where $\exp$ denotes the exponential map in $G$, $p\in M$, and $\cdot$ denotes the (isometric) action of $G$ on $M$. A Riemannian manifold $(M,g)$ is called a \emph{geodesic orbit manifold} (or g.o. manifold) if all geodesics on $M$ are homogeneous; The metric $g$ is called a \emph{geodesic orbit metric}.  The study and classification of g.o. manifolds is an intensively studied problem in Riemannian geometry for the last 35 years (see the recent monograph \cite{BerNik20} on the subject).  It also has extensions to diverse geometric contexts, such as in sub-Riemannian (\cite{Po22}) and Finsler geometry (\cite{YanDe14}). Generalizations of homogeneous geodesics, written as products of multiple exponential factors, were studied in \cite{So23}.

Apart from their geometric simplicity, homogeneous geodesics are physically interpreted as \emph{relative equilibria} of fluids, or \emph{stationary rotations} of rigid bodies (\cite{Arn89}), that is free motions where the particles of the body are relatively motionless. Geodesic orbit metrics are also applied in fields such as \emph{computational anatomy} (\cite{LoPe13}) and the \emph{learning of neural networks} (\cite{NisAk05}), due to the convenient computational implementation of the geodesics \eqref{homgeod} as matrix exponentials.   

Despite the several known classes of g.o. manifolds, such as \emph{symmetric} (\cite{KoNo69}), \emph{weakly symmetric} (\cite{BeKoVa97}) and \emph{Clifford - Wolf} (\cite{BerNik09}) spaces, the classification of g.o. manifolds remains an open problem, even within the class of compact simple Lie groups, where only partial classifications have been produced up to this day (e.g. \cite{CheCheWo18}).  One may attribute the difficulty of the classification problem on Lie groups to the intricacy of the moduli space of left-invariant metrics: essentially any inner product on the corresponding Lie algebra generates such a metric and vice versa.

  Motivated by this problem, we discuss some basic tools and insights (Section \ref{Sec2}), expecting to simplify and initiate a more systematic study of the classification problem of g.o. metrics on compact simple Lie groups, at least of low ranks. At the same time, the main result of this paper is the complete classification of the left-invariant g.o. metrics on the compact simple Lie group $G_2$ (Theorem \ref{main}).

Before we state our main result, let $G$ be a Lie group with Lie algebra $\fr{g}$. Then the left-invariant Riemannian metrics $g$ are in bijection with inner products $\langle \ ,\ \rangle$ on $\fr{g}$. Henceforth assume that $G$ is compact and simple, and let $B$ denote the Killing form of $\fr{g}$ which is negative-definite. Any Lie subalgebra $\fr{k}$ of $\fr{g}$ induces the $B$-orthogonal decomposition 

\begin{equation}\label{dec}\fr{g}=\fr{k}\oplus \fr{m}=\underbrace{\fr{z}(\fr{k})\oplus \fr{k}_1\oplus \cdots \oplus \fr{k}_s}_{\fr{k}}\oplus\fr{m}, \end{equation}

\noindent where $\fr{z}(\fr{k})$ is the center of $\fr{k}$, the subalgebras $\fr{k}_1,\dots,\fr{k}_s$ are the simple ideals of $\fr{k}$, and $\fr{m}$ is the $B$-orthogonal complement of $\fr{k}$ in $\fr{g}$. We consider the class of left-invariant metrics generated by inner products of the form 

\begin{equation}\label{metricform} \langle \ ,\ \rangle=\left.( \ , \ )\right|_{\fr{z}(\fr{k})\times \fr{z}(\fr{k})}+ \lambda_1\left.(-B)\right|_{\fr{k}_1\times \fr{k}_1}+\cdots + \lambda_s\left.(-B)\right|_{\fr{k}_s\times \fr{k}_s}+\mu\left.(-B)\right|_{\fr{m}\times \fr{m}},\end{equation}

\noindent where $\left.( \ , \ )\right|_{\fr{z}(\fr{k})\times \fr{z}(\fr{k})}$ denotes any inner product on $\fr{z}(\fr{k})$. The above metrics are precisely the \emph{naturally reductive} metrics (c.f. \cite{DaZi79} for the terminology) of the compact simple Lie groups.  Our main result shows that the above metrics exhaust all g.o. metrics on $G_2$.

\begin{theorem}\label{main}Let $g$ be a left-invariant Riemannian metric on the compact Lie group $G_2$. Then $g$ is a geodesic orbit metric if and only if there exists a Lie subalgebra $\fr{k}$ of $\fr{g}=\fr{g}_2$ such that the corresponding inner product $\langle \ , \  \rangle$ has the form  \eqref{metricform}. In particular, $g$ is a geodesic orbit metric if and only if it is naturally reductive. \end{theorem} 

Theorem \ref{main} is proved in Section \ref{Sec3}.

\section{Aspects of the study of g.o. metrics on compact simple Lie groups}\label{Sec2}

For the rest of this section, let $G$ be a compact, connected and simple Lie group with Lie algebra $\fr{g}$.  Denote by $B$ the Killing form of $\fr{g}$. 

\subsection{Left-invariant metrics and metric endomorphisms} A Riemannian metric $g$ on $G$ is called left-invariant if all left translations $L_x$, $x\in G$, are isometries of $(G,g)$. The left-invariant metrics $g$ are in bijection with inner products $\langle \ ,\ \rangle$ on $\fr{g}$.  In turn, they are in bijection with \emph{metric endomorphisms} $\Lambda:\fr{g}\rightarrow \fr{g}$, i.e. endomorphisms that are symmetric with respect to $B$, such that

 \begin{equation*}\label{metend}\langle X,Y\rangle = (-B)(\Lambda X, Y), \ \ X,Y\in \fr{g}.\end{equation*}

\noindent Henceforth, we will make no distinction between a left-invariant metric $g$ and its corresponding metric endomorphism $\Lambda$.

\subsection{$\op{Ad}_K$-invariant metrics} Denote by $\op{Ad}:G\rightarrow \op{Gl}(\fr{g})$, $x\mapsto \op{Ad}_x$, the adjoint representation of $G$.  Denote also by $\op{ad}:\fr{g}\rightarrow \fr{gl}(\fr{g})$, $X\mapsto \op{ad}_X$, the adjoint representation of $\fr{g}$.  Let $K$ be a Lie subgroup of $G$ with Lie algebra $\fr{k}$.  A left-invariant metric $g$ is called \emph{$\op{Ad}_K$-invariant} (or right-invariant by $K$) if all right translations by elements of $K$ are also isometries of $(G,g)$. Equivalently, the corresponding metric endomorphism $\Lambda$ is $\op{Ad}_K$-equivariant, that is $\Lambda\circ \op{Ad}_k=\op{Ad}_k\circ \Lambda$ for all $k\in K$.  Consequently, the endomorphism $\Lambda$ is $\op{ad}_{\fr{k}}$-equivariant, while the last condition is equivalent to the $\op{Ad}_K$-invariance of $g$ if $K$ is connected.  Therefore, the form of any $\op{Ad}_K$-invariant metric depends on the structure of the corresponding representation $\left.\op{Ad}\right|_K:K\rightarrow \op{Gl}(\fr{g})$. We will not delve into further details for the explicit form of the metric endomorphism $\Lambda$, but it can easily be deduced using basic representation theory. Generally, the form of $\Lambda$ becomes more complicated if there exist equivalent $\op{Ad}_K$-submodules in $\fr{g}$. 

Assuming that $K$ is connected, consider the decomposition \eqref{dec}.  Suppose that any non-trivial $\op{Ad}_K$-submodule in $\fr{k}$ is inequivalent to any $\op{Ad}_K$-submodule in $\fr{m}$ (or equivalently, any $\op{Ad}_K$-equivariant map between $\fr{k}$ and $\fr{m}$ is trivial).  Then Schur's lemma implies that both spaces $\fr{k}$ and $\fr{m}$ are invariant by $\Lambda$. More specifically, since the simple ideals $\fr{k}_1,\dots,\fr{k}_s$ are $\op{Ad}_K$-irreducible, pairwise inequivalent and inequivalent with the center $\fr{z}(\fr{k})$, the endomorphism $\Lambda$ can be represented by the block-diagonal form

\begin{equation}\label{Schurform}\Lambda=\begin{pmatrix} 
 \left.\Lambda\right|_{\fr{z}(\fr{k})} & 0 & \cdots &0 &0\\
 0& \left.\lambda_1\op{Id}\right|_{\fr{k}_1} &\cdots &0 &0\\
  \vdots & \cdots & \ddots &\vdots &\vdots\\
  0&\cdots &\cdots &\left.\lambda_s\op{Id}\right|_{\fr{k}_s} &0\\
  0&\cdots &\cdots &0 &\left.\Lambda\right|_{\fr{m}}
  \end{pmatrix},
\end{equation}
 
\noindent with respect to the decomposition \eqref{dec}.  Unfortunately, the above diagonal form parametrizes all $\op{Ad}_K$-invariant metrics only in the case where there exist no non-trivial $\op{Ad}_K$-equivariant maps between $\fr{k}$ and its $B$-orthogonal complement $\fr{m}$ (e.g. if $K$ has maximal rank).  However, there exists a large class of algebras $\fr{k}$ for which this diagonalization is possible: These are the normalizers of the weakly regular algebras that we will see in the sequel.

\subsection{Regular and weakly regular subgroups and subalgebras} For a Lie subalgebra $\fr{h}$ of $\fr{g}$, let 

\[ \fr{n}_{\fr{g}}(\fr{h})=\{X\in \fr{g}: [X,Y]\in \fr{h} \ \ \makebox{for all} \ \ Y\in \fr{h}\}\] 

\noindent be the normalizer of $\fr{h}$ in $\fr{g}$, i.e. the largest subalgebra of $\fr{g}$ containing $\fr{h}$ as an ideal.  Following Dynkin (\cite{Dyn52}), a Lie subalgebra $\fr{h}$ of $\fr{g}$ is called \emph{regular} if its normalizer contains a Cartan subalgebra of $\fr{g}$. Any abelian subagebra and any subalgebra of $\fr{g}$ having maximal rank is regular. We note that the regular subalgebras are precisely the ideals of the subalgebras of $\fr{g}$ that have maximal rank. Accordingly, a Lie subgroup $H$ of $G$ is called regular if its Lie algebra $\fr{h}$ is a regular subalgebra of $\fr{g}$. Set $\fr{k}:=\fr{n}_{\fr{g}}(\fr{h})$ and consider the $B$-orthogonal decomposition $\fr{g}=\fr{k}\oplus \fr{m}$.  In \cite{So25}, Lemma 4.1, it is shown that  if $H$ is a regular subgroup of $G$ with Lie algebra $\fr{h}$ then any non-trivial $\op{ad}_{\fr{k}}$-submodule in $\fr{k}$ is inequivalent to any $\op{ad}_{\fr{k}}$-submodule in $\fr{m}$. This property has led to the following definition in \cite{So25}.

\begin{definition}\label{weakly} Let $H$ be a Lie subgroup of $G$ with Lie algebra $\fr{h}$.  Set $\fr{k}:=\fr{n}_{\fr{g}}(\fr{h})$ and consider the $B$-orthogonal decomposition $\fr{g}=\fr{k}\oplus \fr{m}$.  The subgroup $H$ is called weakly regular if any non-trivial $\op{ad}_{\fr{k}}$-submodule in $\fr{k}$ is inequivalent to any $\op{ad}_{\fr{k}}$-submodule in $\fr{m}$. The algebra $\fr{h}$ is also called weakly regular.\end{definition}

\noindent Consequently, any regular subgroup of $G$ is weakly regular, but there exist weakly regular subgroups that are not regular (\cite{So25}). More importantly, if $H$ is a weakly regular subgroup of $G$ then any $\op{ad}_{\fr{k}}$-equivariant endomorphism $\Lambda:\fr{g}\rightarrow \fr{g}$ has the block diagonal form \eqref{Schurform}.

\begin{remark} In view of the above discussion, it would be useful to determine the weakly regular subalgebras of compact simple Lie groups. For example, in the proof of Theorem \ref{main} in the final section, it is shown that any subalgebra of $\fr{g}_2$ is weakly regular. This fact greatly simplifies the study of g.o. metrics on $G_2$.\end{remark}

\subsection{Isometry groups of left-invariant metrics: The Ochiai-Takahashi theorem}  The following useful theorem of Ochiai and Takahashi yields the form of the isometry groups of compact simple Lie groups endowed with left-invariant metrics.

\begin{theorem}\emph{(\cite{OT76})} Let $G$ be a compact, connected, simple Lie group and $g$ be a left-invariant Riemannian metric on $G$. Then the isometry group of $(G,g)$ is contained in $L(G)R(G)$, i.e. the product of the groups of left and right translations in $G$.\end{theorem}  

The above theorem implies that the isometry group of $(G,g)$ is locally isomorphic to $G\times H$, where $H$ is a closed subgroup of $G$.  The group $G\times H$ acts on $G$ by $(x,y)\cdot z:=xzy^{-1}$, $(x,y)\in G\times H$, $z\in G$. As a result, for any left-invariant metric on $G$ there exists a Lie subalgebra $\fr{h}$ of $\fr{g}$ such that the corresponding metric endomorphism $\Lambda$ is $\op{ad}_{\fr{h}}$-equivariant.   If $g$ is a g.o. metric, this means that there exists a Lie subgroup $H$ of $G$ such that any geodesic of $(G,g)$ is the orbit of an one-parameter subgroup of $G\times H$ under the aforementioned action. To emphasize the subgroup $H$, we call $g$ a \emph{$G\times H$-geodesic orbit metric}. It is trivial but useful to note that any $G\times H$-g.o. metric is $\op{Ad}_H$-invariant. In the sequel, we present a condition for $G\times H$-g.o. metrics.

\subsection{The geodesic orbit condition for compact simple Lie groups}  The following is a necessary and sufficient condition for $G\times H$-g.o. metrics on compact Lie groups. It was initially proven for compact simple Lie groups, and under an additional condition in \cite{Nik19}, and then generalized for compact Lie groups in \cite{So25}. We state it for simple Lie groups.  

\begin{prop}\label{gocon}Let $G$ be a compact, connected, simple Lie group endowed with a left-invariant metric $g$ and let $\Lambda:\fr{g} \rightarrow \fr{g}$ be the corresponding metric endomorphism. Then $g$ is a $G\times H$-g.o. metric if and only if for all $X\in \fr{g}$ there exists a vector $W=W(X)$ in the Lie algebra $\fr{h}$ of $H$ such that

\begin{equation*}\label{GOCond}[W+X,\Lambda X]=0.\end{equation*}   \end{prop}

\noindent Working the above condition, one can derive various simplification results for g.o. metrics.  For example, the following simple Lie-algebraic condition allows the reduction of the possible eigenvalues of the metric endomorphism corresponding to a g.o. metric.

\begin{lemma}\label{eigenv} Let $G$ be a compact, connected, simple Lie group endowed with a left-invariant g.o. metric $g$.  Let $\fr{g}_1,\fr{g}_2\subseteq \fr{g}$ be eigenspaces of the associated metric endomorphism $\Lambda$, corresponding to eigenvalues $\lambda_1, \lambda_2$ respectively. If the space $[\fr{g}_1,\fr{g}_2]$ has non-zero $B$-orthogonal projection to the $B$-orthogonal complement of $\fr{g}_1\oplus \fr{g}_2$ then $\lambda_1=\lambda_2$.\end{lemma}

\begin{proof}  There exists a subgroup $H$ of $G$ such that $g$ is a $G\times H$-g.o. metric. Let $X_i\in \fr{g}_i$, $i=1,2$, and set $X:=X_1+X_2$.  By Proposition \ref{gocon}, there exists a vector $W\in \fr{h}$ such that $0=[W+X,\Lambda X]=\lambda_1[W,X_1]+\lambda_2[W,X_2]-(\lambda_1-\lambda_2)[X_1,X_2]$, i.e.

\[  (\lambda_1-\lambda_2)[X_1,X_2]=\lambda_1[W,X_1]+\lambda_2[W,X_2].\]

\noindent Since the metric $g$ is $\op{Ad}_H$-invariant, each eigenspace $\fr{g}_i$ of $\Lambda$ is $\op{ad}_{\fr{h}}$-invariant.  Therefore, the above equation yields $(\lambda_1-\lambda_2)[X_1,X_2]\in \fr{g}_1\oplus \fr{g}_2$ for all $X_1\in \fr{g}_1$, $X_2\in \fr{g}_2$, implying that $(\lambda_1-\lambda_2)[\fr{g}_1,\fr{g}_2]\subseteq \fr{g}_1\oplus \fr{g}_2$.  Then the assumption of the lemma yields $\lambda_1=\lambda_2$.\end{proof}

\begin{example}Let $G=Sp(n)$ and $K=Sp(n_1)\times Sp(n_2)\times Sp(n_3)\subseteq G$, where $n_1+n_2+n_3=n$.  We will classify the $G\times K$-g.o. metrics on $G$.  We consider the Lie algebra $\fr{k}=\fr{sp}(n_1)\oplus \fr{sp}(n_2)\oplus \fr{sp}(n_3)$ of $K$.  Then we have the $B$-orthogonal decomposition

\begin{equation}\label{far}\fr{g}=\fr{k}\oplus \fr{m}=\fr{sp}(n_1)\oplus \fr{sp}(n_2)\oplus \fr{sp}(n_3)\oplus \fr{m}_1\oplus \fr{m}_2\oplus \fr{m}_3, \end{equation}

\noindent where $\fr{m}_i$, $i=1,2,3$, are irreducible and pairwise inequivalent $\op{ad}_{\fr{k}}$-submodules in $\fr{m}$, satisfying the relations

\begin{equation}\label{GenWal}[\fr{m}_i,\fr{m}_j]\subseteq \fr{m}_k, \ \ \makebox{$i,j,k$ distinct}\end{equation}
 
 \noindent (see e.g. \cite{CheKaLi16}, \cite{Nik16}, \cite{Nik21}, where the corresponding class of spaces $G/K$, called generalized Wallach spaces, were independently classified).  Since all the $\op{ad}_{\fr{k}}$-submodules  in decomposition \eqref{far} are irreducible and pairwise inequivalent, any g.o. metric $\Lambda$ has the block-diagonal form \eqref{Schurform}, where $\left.\Lambda\right|_{\fr{m}}=\begin{pmatrix}\mu_1\left.\op{Id}\right|_{\fr{m}_1} & 0 & 0\\
0 & \mu_2\left.\op{Id}\right|_{\fr{m}_2} & 0\\
0 & 0 & \mu_3\left.\op{Id}\right|_{\fr{m}_3}\end{pmatrix}$.  On the other hand, relation \eqref{GenWal} implies that $[\fr{m}_i,\fr{m}_j]$ has non-zero $B$-orthogonal projection to the $B$-orthogonal complement of $\fr{m}_i\oplus \fr{m}_j$, $i=1,2,3$. Then Lemma \ref{eigenv} implies that $\mu_1=\mu_2=\mu_3$.  We conclude that any g.o. metric has the form 

\begin{equation*}\Lambda=\begin{pmatrix} 
\left.\lambda_1\op{Id}\right|_{\fr{sp}(n_1)}& 0  &0 &0\\
 0&\left.\lambda_2\op{Id}\right|_{\fr{sp}(n_2)}&0 &0\\
  0 & 0 &\left.\lambda_3\op{Id}\right|_{\fr{sp}(n_3)}&0\\
  0&0 &0 &\left.\mu\op{Id}\right|_{\fr{m}}
  \end{pmatrix},\end{equation*}
  
  \noindent i.e. corresponds to an inner product of the form \eqref{metricform}.  On the other hand, any metric of this form is naturally reductive (\cite{DaZi79}) and thus geodesic orbit (\cite{KoNo69}).  Therefore, the $G\times K$-g.o. metrics on $Sp(n)$ are precisely the metrics of the form \eqref{metricform}, with $\fr{k}=\fr{sp}(n_1)\oplus \fr{sp}(n_2)\oplus \fr{sp}(n_3)$.  We remark that the above classification can be derived immediately as a special case of Theorem 3.1 in \cite{CheCheDe18}. \end{example}

\subsection{The normalizer lemma and the splitting theorem} As a generalization of a similar result of Nikonorov for g.o. metrics on homogeneous spaces (\cite{Nik17}), the following is true for g.o. metrics on compact Lie groups.  

\begin{lemma}\label{normalizerlemma}\emph{(\cite{So25})} Let $G$ be a compact Lie group with Lie algebra $\fr{g}$ and let $g$ be a $G\times H$-g.o. metric on $G$ with corresponding metric endomorphism $\Lambda$. Denote by $\fr{h}$ the Lie algebra of $H$. Then $\Lambda$ is not only $\op{ad}_{\fr{h}}$-equivariant, but also $\op{ad}_{\fr{k}}$-equivariant, where $\fr{k}=\fr{n}_{\fr{g}}(\fr{h})$.\end{lemma}

The following is an immediate corollary of Lemma \ref{normalizerlemma}.

\begin{corol}Let $(G, g)$ be a compact simple Riemannian Lie group where $g$ is
a left-invariant geodesic orbit metric. Let $\fr{g}\times \fr{k}$ be the Lie algebra of the isometry group of $(G, g)$. Then $\fr{k}$ is a self-normalizing subalgebra of $\fr{g}$, i.e. $\fr{k}=\fr{n}_{\fr{g}}(\fr{k})$.\end{corol}

 The normalizer lemma is quite useful in the case where $H$ is a weakly regular subgroup of $G$, as it induces the following splitting of the g.o. metric.

\begin{theorem}\label{splittingtheorem}\emph{(\cite{So25})} Let $G$ be a compact Lie group with Lie algebra $\fr{g}$ and let $H$ be a connected weakly regular subgroup of $G$ with Lie algebra $\fr{h}$. Set $\fr{k}:=\fr{n}_{\fr{g}}(\fr{h})$ and consider an orthogonal decomposition $\fr{g}=\fr{k}\oplus \fr{m}$ with respect to an $\op{Ad}$-invariant inner product on $\fr{g}$. Let $\Lambda$ be the metric endomorphism corresponding to a left-invariant $G\times H$-g.o. metric g on $G$.  Then the spaces $\fr{k}$ and $\fr{m}$ are $\Lambda$-invariant, i.e. $\Lambda$ splits as $\Lambda=\begin{pmatrix}\left.\Lambda\right|_{\fr{k}} & 0\\ 0 & \left.\Lambda\right|_{\fr{m}}\end{pmatrix}$.  In particular, $\Lambda$ can be represented by the block-diagonal form \eqref{Schurform}, where the restriction $\left.\Lambda\right|_{\fr{m}}$ defines a $G$-g.o. metric on the homogeneous space $G/H$. \end{theorem}

\begin{remark}\label{classifications}Under the assumptions of Theorem \ref{splittingtheorem},  the restriction $\left.\Lambda\right|_{\fr{m}}$ defines a g.o. metric on the homogeneous space $G/H$.  This fact reduces, to a large extent, the classification of g.o. metrics on Lie groups to the classification of g.o. metrics on homogeneous spaces $G/H$.  There exists a plethora of partial classifications for compact homogeneous g.o. spaces $(G/H,g)$.  Some notable partial classifications (for example if $H$ is abelian or simple) can be found in the works \cite{AlArv07}, \cite{AlNik09}, \cite{CheNiNi23}, \cite{So21}, \cite{Ta99}.  Homogeneous g.o. spaces are intensively studied, but we will not delve further into this topic.  Instead, we refer the interested reader to the monograph \cite{BerNik20} for results up to 2020.  \end{remark}

\begin{remark}\label{remnormal}If $G$ is simple and $\left.\Lambda\right|_{\fr{m}}=\lambda\left.\op{Id}\right|_{\fr{m}}$ then $\left.\Lambda\right|_{\fr{m}}$ is the standard metric, induced from the Killing form of $\fr{g}$, which is geodesic orbit.\end{remark}

\section{Geodesic orbit metrics on $G_2$ - Proof of Theorem \ref{main}}\label{Sec3}

Using the main tools discussed in the previous section, we prove Theorem \ref{main}. Firstly we need the following.

\begin{lemma}\label{lemnorm}Let $\fr{g}$ be a compact Lie algebra of rank two and let $\fr{h}$ be a non-regular subalgebra of $\fr{g}$.  Then $\fr{h}$ is self-normalizing in $\fr{g}$, i.e. $\fr{h}=\fr{n}_{\fr{g}}(\fr{h})$.\end{lemma}

\begin{proof} Since the trivial subalgebra is regular, we have $\fr{h}\neq \{0\}$ and thus $rank(\fr{h})\geq 1$. In fact, since $\fr{h}$ is not regular, it cannot have maximal rank, and thus $rank(\fr{h})=1$ and $\fr{h}$ is simple.  We consider the Lie algebra direct sum $\fr{n}_{\fr{g}}(\fr{h})=\fr{h}\oplus \fr{c}_{\fr{g}}(\fr{h})$, where $\fr{c}_{\fr{g}}(\fr{h})=\{X\in \fr{g}:[X,Y]=0 \ \ \makebox{for all} \ \ Y\in \fr{h}\}$.  If $X$ is a non-zero vector in $\fr{c}_{\fr{g}}(\fr{h})$  and $\fr{t}$ is a (one-dimensional) Cartan subalgebra of $\fr{h}$ then $\fr{t}\oplus \op{span}_{\mathbb R}\{X\}$ is an abelian subalgebra of $\fr{n}_{\fr{g}}(\fr{h})$ having rank two. Therefore, the rank of $\fr{n}_{\fr{g}}(\fr{h})$ is two, which is a contradiction given that $\fr{h}$ is not regular.  We conclude that $X=\{0\}$, implying that $\fr{c}_{\fr{g}}(\fr{h})=\{0\}$, and thus $\fr{n}_{\fr{g}}(\fr{h})=\fr{h}$.  \end{proof}

\emph{Proof of Theorem \ref{main}.}  Firstly, we recall again that any metric of the form \eqref{metricform} is naturally reductive (\cite{DaZi79}) and hence geodesic orbit.  Conversely, assume that $\Lambda$ is the metric endomorphism of a left-invariant g.o. metric on $G_2$.  Then there exists a Lie subgroup $H$ of $G_2$ such that $\Lambda$ is a $G_2\times H$-g.o. metric, i.e. every geodesic is the orbit of an one-parameter subgroup of $G_2\times H$. Let $\fr{h}$ be the Lie algebra of $H$.  Since the one-parameter subgroups of $G_2\times H$ have the form $(\exp(tX),\exp(tY))$, $X \in \fr{g}_2$, $Y\in \fr{h}$, we may assume that $H=\exp(\fr{h})$, i.e. $H$ is connected, without any loss of generality.

For the next step, we will show that every subalgebra of $G_2$ is weakly regular, which implies that $H$ is a weakly regular subgroup of $G_2$.  According to \cite{Ma16}, the Lie subalgebras of the complexified Lie algebra $\fr{g}_2^{\mathbb C}$ are exhausted (up to conjugation by an inner automorphism) by 64 types of regular subalgebras, 2 types of non-regular semisimple subalgebras and 49 types of non-regular solvable subalgebras. Therefore, the compact real form $\fr{g}_2$ has only 2 non-regular subalgebras. Although they are not regular, we will prove, based on the results in \cite{So22}, that those two subalgebras are weakly regular. 

To list those algebras, we firstly consider a Cartan subalgebra $\fr{t}$ of $\fr{g}_2^{\mathbb C}$ and the root decomposition $\fr{g}_2^{\mathbb C}=\fr{t}\oplus \sum_{\gamma\in R}\fr{g}^{\gamma}$, corresponding to the root system 

\[  R=\{\pm \alpha, \pm \beta, \pm (\alpha+\beta),\pm (2\alpha+\beta), \pm (3\alpha+\beta),\pm (3\alpha+2\beta)\} \]

\noindent of $\fr{g}_2^{\mathbb C}$, where $\fr{g}^{\gamma}=\{X\in \fr{g}_2^{\mathbb C}:[a,X]=\gamma(a)X \ \ \makebox{for all} \ \ a\in \fr{t}\}$.  Denote by $Q$ the Killing form of $\fr{g}_2^{\mathbb C}$, which is non-degenerate on $\fr{t}$, and let $( \ , \ )$ be the dual form on $\fr{t}^*$.  For any $\gamma\in R\subset \fr{t}^*$, let $t_{\gamma}$ be the corresponding covector defined by $Q(t_{\gamma},a)=\gamma(a)$ for all $a\in \fr{t}$.  Set $H_{\gamma}:=\frac{2t_{\gamma}}{(\gamma,\gamma)}$.   Consider root elements $E_{\gamma}\in \fr{g}^{\gamma}$ so that the set $\{H_{\alpha},H_{\beta},E_{\gamma}: \gamma\in R\}$ is a Chevalley basis of $\fr{g}^{\mathbb C}$. Let $R^+$ be the set of positive roots in $R$. For $\gamma\in R^+$, set 

\begin{equation*}F_{\gamma}:=E_{\gamma}-E_{-\gamma},\ \ G_{\gamma}:=\sqrt{-1}(E_{\gamma}+E_{-\gamma})\ \  \makebox{and}  \ \ \fr{m}_{\gamma}:=\op{span}_{\mathbb R}\{F_{\gamma},G_{\gamma}\}.\end{equation*}

\noindent Then the compact real form $\fr{g}_2$ of $\fr{g}_2^{\mathbb C}$ admits the $B$-orthogonal decomposition 

\[ \fr{g}_2=\sqrt{-1}\fr{t}\oplus \bigoplus_{\gamma\in R^+}\fr{m}_{\gamma}. \]

\noindent According to \cite{Ma16} (see also \cite{DoRe18}), the two non-regular semisimple subalgerbas of $\fr{g}_2^{\mathbb C}$ are the subalgebras $\fr{h}^{\mathbb C}_1=\op{span}_{\mathbb C}\{\sqrt{2}(E_{3\alpha+2\beta}+E_{-\beta}),\sqrt{2}(E_{\beta}+E_{-(3\alpha+2\beta)}),2H_{3\alpha+\beta}\}$ and $\fr{h}^{\mathbb C}_2=\op{span}_{\mathbb C}\{\sqrt{6}E_{\alpha}+\sqrt{10}E_{\beta},\sqrt{6}E_{-\alpha}+\sqrt{10}E_{-\beta},14H_{9\alpha+5\beta}\}$, isomorphic to $\fr{sl}_2\mathbb C$. Their corresponding compact real forms in $\fr{g}_2$ are the algebras $\fr{h}_1=\op{span}_{\mathbb R}\{\sqrt{2}(F_{3\alpha+2\beta}-F_{\beta}),\sqrt{2}(G_{3\alpha+2\beta}+G_{\beta}),2\sqrt{-1} H_{3\alpha+\beta}\}$ and $\fr{h}_2=\op{span}_{\mathbb R}\{\sqrt{6}F_{\alpha}+\sqrt{10}F_{\beta},\sqrt{6}G_{\alpha}+\sqrt{10}G_{\beta},14\sqrt{-1} H_{9\alpha+5\beta}\}$, both isomorphic to $\fr{su}(2)$.  Moreover, by Lemma \ref{lemnorm}, both Lie subalgebras $\fr{h}_i$ are equal to their corresponding normalizers $\fr{n}_{\fr{g}_2}(\fr{h}_i)$.  

 According to \cite{So22}, the $B$-orthogonal complement $\fr{m}$ of $\fr{h}_1$ in $\fr{g}_2$ decomposes into $\op{ad}_{\fr{h}_1}$-irreducible (and thus $\op{ad}_{\fr{n}_{\fr{g}_2}(\fr{h}_1)}$-irreducible) submodules as $\fr{m}=\fr{p}\oplus \fr{q}$, where \\$\fr{p}=\op{span}_{\mathbb R}\{\sqrt{2}(F_{3\alpha+2\beta}+F_{\beta}),\sqrt{2}(G_{3\alpha+2\beta}-G_{\beta}), 2\sqrt{-1} H_{\alpha+\beta}\}\oplus \fr{m}_{3\alpha+\beta}$ and $\fr{q}=\fr{m}_{\alpha}\oplus \fr{m}_{\alpha+\beta}\oplus \fr{m}_{2\alpha+\beta}$.  Since $\fr{h}_1$ is simple and self-normalizing, it is $\op{ad}_{\fr{h}_1}$-irreducible and thus $\op{ad}_{\fr{n}_{\fr{g}_2}(\fr{h}_1)}$-irreducible. Comparing the dimensions of the irreducible submodules $\fr{p}$ and $\fr{q}$ with $\fr{h}_1$, we deduce that any non-trivial $\op{ad}_{\fr{n}_{\fr{g}_2}(\fr{h}_1)}$-submodule in $\fr{h}_1=\op{ad}_{\fr{n}_{\fr{g}_2}(\fr{h}_1)}$ is inequivalent to any $\op{ad}_{\fr{n}_{\fr{g}_2}(\fr{h}_1)}$-submodule in $\fr{m}$.  By definition \ref{weakly}, $\fr{h}_1$ is a weakly regular subalgebra of $\fr{g}_2$.

Similarly, according to \cite{So22}, the $B$-orthogonal complement of $\fr{h}_2$ in $\fr{g}_2$ is the space $\fr{m}=\op{span}_{\mathbb R}\{\sqrt{10}F_{\alpha}-3\sqrt{6}F_{\beta},\sqrt{10}G_{\alpha}-3\sqrt{6}G_{\beta},2\sqrt{-1} H_{\alpha-\beta}\}\oplus \fr{m}_{\alpha+\beta}\oplus \fr{m}_{2\alpha+\beta}\oplus \fr{m}_{3\alpha+\beta}\oplus \fr{m}_{3\alpha+2\beta}$.  This space is $\op{ad}_{\fr{h}_2}$-irreducible and thus $\op{ad}_{\fr{n}_{\fr{g}_2}(\fr{h}_2)}$-irreducible.  Using the same reasoning as above, we deduce that $\fr{h}_2$ is also a weakly regular subalgebra of $\fr{g}_2$.  We conclude that both $\fr{h}_1$ and $\fr{h}_2$ are weakly regular subalgebras of $\fr{g}_2$, and thus all subalgebras of $\fr{g}_2$ are weakly regular.

Summarizing the above, we conclude that any g.o. metric on $G_2$ is a $G_2\times H$-g.o. metric with $H$ being a connected weakly regular subgroup.  By Theorem \ref{splittingtheorem}, the metric endomorphism $\Lambda$ can be represented by the block-diagonal form \eqref{Schurform}, where the restriction $\left.\Lambda\right|_{\fr{m}}$ defines a $G_2$-g.o. metric on the homogeneous space $G_2/H$.  On the other hand, by Corollary 1.2 in \cite{So22}, any such metric is standard, i.e. $\Lambda=\left.\op{Id}\right|_{\fr{m}}$.  Therefore, the metric $\Lambda$ corresponds to the inner product \eqref{metricform}, thus concluding the proof.\qed

\end{document}